\documentclass{amsart}
\usepackage{amsmath}
\usepackage{amsfonts}
\usepackage{mathrsfs}
\usepackage{amssymb}
\usepackage{amsthm}
\usepackage{verbatim}
\usepackage{epic}
\usepackage{eepic}
\newtheorem{theorem}{Theorem}[section]
\newtheorem{lemma}[theorem]{Lemma}
\newtheorem{corollary}[theorem]{Corollary}
\newtheorem{proposition}[theorem]{Proposition}
\newtheorem*{maintheorem}{Main Theorem}
\theoremstyle{definition}
\newtheorem{definition}[theorem]{Definition}

\newtheorem{property}[theorem]{Property}
\theoremstyle{remark}

\numberwithin{equation}{section}
\newcommand{\BBR}{\mathbb R}
\newcommand{\BBN}{\mathbb N}
\newcommand{\BBC}{\mathbb C}

\newcommand{\T}{\operatorname T}
\newcommand{\M}{\operatorname M}
\newcommand{\SO}{\operatorname{SO}}
\newcommand{\GL}{\operatorname{GL}}
\newcommand{\vectxi}{\overrightarrow{\xi}}
\newcommand{\vectv}{{\overrightarrow{v}}}
\newcommand{\vectw}{{\overrightarrow{w}}}
\newcommand{\vectp}{{\overrightarrow{p}}}
\newcommand{\Drn}{{\widetilde D_{r,n}}}
\newcommand{\Dr}{{\widetilde D_r}}

\begin{document}
\author[S.\ Z.\ Gautam]{S.\ Zubin Gautam}
\address{Department of Mathematics, UCLA, Los Angeles, CA 90095-1555, USA.}
\email{sgautam@ucla.edu}%
\title{On curvature and the bilinear multiplier problem}
\begin{abstract}
We provide sufficient normal curvature conditions on the boundary of
a domain $D \subset \BBR^4$ to guarantee unboundedness of the
bilinear Fourier multiplier operator $\T_D$ with symbol $\chi_D$
outside the local $L^2$ setting, \textit{i.e}.\ from $L^{p_1}
(\BBR^2) \times L^{p_2} (\BBR^2) \rightarrow L^{p_3'} (\BBR^2)$ with
$\sum \frac{1}{p_j} = 1$ and $p_j <2$ for some $j$. In particular,
these curvature conditions are satisfied by any domain $D$ that is
locally strictly convex at a single boundary point.
\end{abstract}
\maketitle
\section{Introduction}
The celebrated ball multiplier theorem of C. Fefferman
(\cite{fefferman}) states that the characteristic function of the
unit ball $B_d$ in $\BBR^d$, $d\geq 2$, is not a bounded Fourier
multiplier on $L^p(\BBR^d)$ for $p\neq 2$.  As an immediate
corollary of the proof, one obtains the corresponding result with
the ball replaced by any connected domain $D$ in $\BBR^d$ whose
boundary is a sufficiently smooth hypersurface with nonzero second
fundamental form (or equivalently a nonzero principal curvature) at
some point.

Interest has arisen in studying analogues of the ball multiplier
question in the bilinear setting; namely, given a domain $D \subset
\BBR^{2d}$, one may ask whether the bilinear Fourier multiplier
$\T_D: \mathcal S(\BBR^d) \times \mathcal S (\BBR^d) \rightarrow
\mathcal S'(\BBR^d)$ given by
\[\T_D\big(f,g\big)(x) = \int_{\BBR^d} \int_{\BBR^d} \chi_D(\xi, \eta) \hat
f(\xi) \hat g(\eta) e^{2\pi i (\xi + \eta) \cdot x} \, \mathrm
d\xi\, \mathrm d\eta\] extends to a bounded bilinear operator from
$L^p (\BBR^d)\times L^q(\BBR^d)$ to $L^r(\BBR^d)$ for suitable
ranges of $p,q,r$; here $\chi_D$ denotes the characteristic function
of $D$. For dimension $d=1$, the case of $D = B_{2}$ the unit disc
of $\BBR^{2}$ was treated by Grafakos and Li, who showed in
\cite{grafakosli} that in fact $\T_{B_{2}}$ \emph{is} a bounded
operator from $L^p(\BBR) \times L^q(\BBR)$ to $L^r(\BBR)$ in the
local $L^2$ case (\textit{i.e}.\ $\frac{1}{p} + \frac{1}{q} +
\frac{1}{r'} = 1$ with $p,q,r'\geq2$). However, the status of the
bilinear disc multiplier \emph{outside} the local $L^2$ case remains
unknown as of this writing, and for the majority of this paper we
concern ourselves only with dimensions $d \geq 2$.

In the linear setting, Fefferman's theorem and the boundedness of
the Hilbert transform give the following dichotomy:  ``Polyhedral''
domains (with finitely many faces) yield bounded Fourier
multipliers, while domains whose boundaries possess curvature (or
simply a suitably rich collection of tangent hyperplanes) give rise
to unbounded multipliers, as noted above. By contrast, the situation
is less well-understood in the bilinear setting, largely because the
boundedness properties of even the half-space multiplier operators
$\T_{P_{\vectv}}$ are not well-understood for $d\geq 2$; here
$P_{\vectv} = \{\vectxi \in \BBR^{2d} \; | \; \vectxi \cdot \vectv >
0\}$ (Theorem \ref{halfplanetheorem} below provides an unboundedness
result for $\T_{P_\vectv}$ in a rather limited range of exponents).
These operators are essentially higher-dimensional analogues of the
bilinear Hilbert transform and are of independent interest (see
Section \ref{notations}; see also \cite{demeterthiele} for a
discussion of related operators and connections to ergodic theory).

Nonetheless, in high dimensions the ideas of Fefferman's original
argument have been successfully adapted to yield unboundedness
results for bilinear Fourier multipliers associated to domains with
boundary curvature.  For $d\geq 2$ and $B_{2d}$ the unit ball of
$\BBR^{2d}$, Diestel and Grafakos (\cite{diestelgrafakos}) proved
that $\T_{B_{2d}}$ is not a bounded bilinear Fourier multiplier
\emph{outside} the local $L^2$ setting; in \cite{grafakosreguera},
Grafakos and Reguera generalized this result to replace the ball
$B_{2d}$ with a compact, strictly convex domain $D$ whose boundary
is a smooth hypersurface in $\BBR^{2d}$.

For both the statements and the proofs of our results, we will adopt
a symmetric presentation in terms of trilinear forms rather than
bilinear operators; this approach rids us of the inconvenience of
dealing with duality, and more importantly it has the decided
benefit of placing our curvature conditions below in a natural
geometric setting.  For now we restrict our attention to dimension
$d=2$; see Remark (2) of Section \ref{remarks} for a discussion of
higher dimensions. To any bilinear operator $\T: \mathcal S(\BBR^2)
\times \mathcal S(\BBR^2) \rightarrow \mathcal S'(\BBR^2)$ we can
associate a trilinear form $\Lambda$ on $\mathcal S(\BBR^2) \times
\mathcal S(\BBR^2) \times \mathcal S(\BBR^2)$ defined by \[ \Lambda
(f_1,f_2,f_3) = \int_{\BBR^2} \T(f_1,f_2)(x) f_3(x) \, \mathrm dx.\]
For triples $\overrightarrow p = (p_1,p_2,p_3)$ with $1 \leq p_j
\leq\infty$ for all $j$, the boundedness of $\T$ from $L^{p_1}
\times L^{p_2}$ to $L^{p_3'}$ is equivalent to the boundedness of
$\Lambda$ on $L^{p_1} \times L^{p_2} \times L^{p_3}$:
\[ |\Lambda(f_1,f_2,f_3)| \leq \|\T\| \,
\prod_{j=1}^3\|f_j\|_{p_j}\,;\] in this case, we say that the form
$\Lambda$ is of ``type $\overrightarrow p$.''  The natural range of
exponent triples $\vectp$ under consideration is given by demanding
that the trivial form $\Lambda_0 (f_1,f_2,f_3) := \int f_1 f_2 f_3$
be bounded; \textit{viz.}, we consider only ``homogeneous'' $\vectp$
with $\sum\frac{1}{p_j} = 1$.

For the bilinear Fourier multiplier operators $\T_D$ as above, the
associated trilinear forms are given by embedding $D$ into $\BBR^6$
as follows: Let $\Gamma$ be the subspace
\[\Gamma:=\{(\xi_1,\xi_2,\xi_3) \in \BBR^2 \times
\BBR^2 \times \BBR^2\; | \;\xi_1 + \xi_2 +\xi_3=0\}\subset \BBR^6,\]
with $\Phi: \BBR^2 \times \BBR^2 \rightarrow \Gamma$ the obvious
isomorphism given by \[\Phi (\xi_1,\xi_2) = \big (\xi_1, \xi_2,
-(\xi_1 + \xi_2)\big).\] Then for $D \subset \BBR^4$, the trilinear
form associated to $\T_D$ is
\[\Lambda_{\Phi(D)}(f_1,f_2,f_3) = \iiint
\delta(\xi_1 + \xi_2 + \xi_3)\chi_{\Phi(D)}(\xi_1,\xi_2,\xi_3) \,
\prod_{j=1}^3 \widehat{f_j}(\xi_j)\, \mathrm d\xi_1 \mathrm d\xi_2
\mathrm d\xi_3.\]  In the
sequel, we will always identify $\BBR^4$ with $\BBR^2 \times \BBR^2$
and $\BBR^6$ with $\BBR^2 \times \BBR^2 \times \BBR^2$, and by a
``$j$-th coordinate slice'' in $\BBR^6$ we mean a $4$-plane of the
form
\[\{(\xi_1,\xi_2,\xi_3) \in \BBR^6 \; | \; \xi_j = \xi_0\}\] for
some fixed $\xi_0 \in \BBR^2$. Note that the intersection of
$\Gamma$ with any $j$-th coordinate slice is a $2$-plane.

Our main result is the following:
\begin{maintheorem}\nonumber \label{mainthm}
Let $\widetilde D$ be a domain in $\Gamma \subset \BBR^6$ such that
$\partial \widetilde D \cap U$ is a smooth, connected
(three-dimensional) hypersurface for some open neighborhood $U
\subset \Gamma$. Suppose that for some $j\in \{1,2,3\}$ the
intersection of $\partial \widetilde D \cap U$ with some $j$-th
coordinate slice is a plane curve of nonzero curvature. Then the
trilinear form $\Lambda_{\widetilde D}$ fails to be of type $\vectp
= (p_1,p_2,p_3)$ whenever $\frac{1}{p_1} + \frac{1}{p_2} +
\frac{1}{p_3} = 1$, $1<p_1,p_2,p_3<\infty$, and $p_i <2$ for some $i
\neq j$.
\end{maintheorem}

Of course, this can be translated to a direct statement about
bilinear multiplier operators associated to domains $D \subset
\BBR^4$ by applying the theorem to $\Phi(D)$.  First and second
coordinate slices in $\Gamma$ correspond to their natural analogues
in $\BBR^4$, while third coordinate slices correspond to $2$-planes
of the form $\{(\xi_1, \xi_2) \in \BBR^4 \; | \; \xi_1 + \xi_2 =
\textrm{constant} \}$.  Since any strictly convex set $D$ is easily
seen to satisfy all three of the given curvature conditions, we
obtain the following generalization of the Grafakos--Reguera result:
\begin{corollary}\label{convexcor}
Let $D$ be a domain in $\BBR^4$ whose boundary $\partial D$ is
smooth in some neighborhood $U \subset \BBR^4$, and suppose that
either $D$ or $\BBR^4 \setminus D$ is strictly convex in this
neighborhood. Then for $\frac{1}{p_1}+ \frac{1}{p_2} + \frac{1}{p_3}
= 1$ with exactly one of $p_1,p_2,p_3$ less than $2$, $\T_D$ does
not extend to a bounded bilinear Fourier multiplier from
$L^{p_1}(\BBR^2) \times L^{p_2}(\BBR^2)$ to $L^{p_3'}(\BBR^2)$.
\end{corollary}

The Main Theorem above can actually be stated more generally;
namely, under essentially the same hypotheses on $\widetilde D$ one
can also prove unboundedness of the operator $\T_D$ outside the
interior of the ``Banach triangle'' (\textit{i.e}.\ from $L^{p_1}
\times L^{p_2}$ to $L^{p_3'}$ with $p_3' \leq 1$, so that
$p_3=\infty$ or $p_3<0$). However, a symmetric statement in terms of
trilinear forms presents some difficulties, as the ``type $\vectp$''
formalism breaks down outside the Banach triangle; to avoid
introducing potentially confusing technicalities at this point, we
defer the statement of the general result to Section \ref{nonbanach}
below.

The idea of the proof of the Main Theorem is quite simple;
heuristically speaking, our approach is simply to apply Fefferman's
original argument on the appropriate coordinate slice.  More
specifically, assuming boundedness of $\Lambda_{\widetilde D}$, we
first obtain a square function estimate for a family of trilinear
forms associated to a family of half-spaces in $\Gamma$.  To
complete the proof, we produce a Besicovitch set-based
counterexample to this estimate.

At this point, given the ease with which one can apply Fefferman's
argument for the ball to more general domains in the linear setting,
the reader may be skeptical as to the necessity of any further
discussion once one has established the unboundedness of the ball
bilinear multiplier.  The key feature distinguishing the bilinear
multiplier problem from the linear one here is a marked decrease in
symmetry with respect to actions on the underlying Euclidean space.
More specifically, the class of linear $L^p$-Fourier multipliers on
$\BBR^d$ is invariant under the natural action of the isometry group
$\operatorname O_d(\BBR) \ltimes \BBR^d$ of $\BBR^d$ (and in fact
under the full affine group $\GL_d(\BBR) \ltimes \BBR^d$); however,
the class of \emph{bilinear} $L^p\times L^q \rightarrow L^r$-Fourier
multipliers on $\BBR^d$ is not invariant under the usual rotation
action of $\SO_{2d}(\BBR)$.\footnote{This fact is well-known but
seems to be folkloric; it can readily be observed by rotating the
symbols of suitable operators falling under the scope of Lemma 1 of
\cite{grafakosli}.  We provide yet another illustrative example in
the discussion following Proposition \ref{muscaluprop} below.} In
fact, it is precisely this absence of $\SO_{2d}$-invariance that
prevents the proofs in \cite{diestelgrafakos} and
\cite{grafakosreguera} from extending easily to more general domains
with curvature.  To wit, though neither proof genuinely requires
surjectivity of the Gauss map $\operatorname N: \partial D
\rightarrow S^{2d-1}$ (which is guaranteed by compactness and strict
convexity), both arguments rely on the presence of a suitable
collection of ``projectively diagonal'' normal vectors of the form
$(v,\lambda v) \in \BBR^d \times \BBR^d$ (see Theorem
\ref{grafakostheorem} below). With such an approach, the ball
appears to be a less generic example in the bilinear setting than in
the linear one; in order to obtain a result treating more general
domains that are merely strictly convex in a neighborhood of some
arbitrary boundary point, one should avoid appealing to the full
wealth of normal directions available on the sphere.

One might expect an ``ideal'' bilinear analogue of Fefferman's
theorem (phrased in terms of operators rather than forms) to state
that $\T_D$ is unbounded for any $D \subset \BBR^4$ whose boundary
has some nonzero principal curvature at a point.  However, the
aforementioned loss of symmetry makes such an analogy impossible:
\begin{proposition}\label{muscaluprop}
There exist domains $\widetilde D \subset \Gamma$ with smooth
boundary such that $\partial \widetilde D$ has nontrivial second
fundamental form at some point while $\Lambda_{\widetilde D}$
\emph{is} of type $(p_1,p_2,p_3)$ whenever $1 < p_1,p_2,p_3 <
\infty$.
\end{proposition}
Examples of such $\widetilde D$ are easily given by certain cylinder
sets: A result of Muscalu (\textit{viz}.\ Theorem 2.1.1 of
\cite{muscalu}) gives the existence of domains $D_0 \subset \BBR^2$
with nontrivial boundary curvature for which the bilinear operators
$\T_{D_0}$ are bounded from $L^{p_1}(\BBR) \times L^{p_2}(\BBR)$ to
$L^{p_3'}(\BBR)$ for all triples $(p_1,p_2,p_3)$ as in the
Proposition.  From any such $D_0$, we construct the domain
\[D = \{(\xi_1, \xi_2, \xi_3, \xi_4) \in \BBR^4 \; | \; (\xi_1,
\xi_3) \in D_0\}\] (here we have of course broken with our
convention of always writing $\BBR^4$ as $\BBR^2 \times \BBR^2$).
Then it is easy to check that \[\T_D \big(f_1, f_2\big)(x_1, x_2) =
\T_{D_0}\big(f_1(\, \cdot \, , x_2) \, , \, f_2( \,\cdot \, , x_2)
\big) (x_1),\] so that boundedness of $\T_{D_0}$ gives the desired
boundedness of $\T_D$ and the usual associated trilinear form
$\Lambda_{\Phi (D)}$.  These ``degenerate'' examples of course
illustrate the aforementioned anisotropy of the bilinear setting;
their curvature is restricted to planes of the form $\{\vectxi \in
\BBR^4 \; | \; (\xi_2, \xi_4) = (a,b)\}$ for fixed $(a,b) \in
\BBR^2$, and they can be suitably rotated to fall under the scope of
the Main Theorem.

\subsection*{Acknowledgements}
I would like to thank my advisor Christoph Thiele for many useful
discussions and in particular for suggesting the symmetric treatment
in terms of trilinear forms, which greatly improved the exposition
of this paper.

\section{Notations and preliminaries}\label{notations}
As in the linear case treated by Fefferman, the key feature
obstructing boundedness of $\T_D$ (or of $\Lambda_{\Phi (D)}$) is
the fact that $\partial D$ (or $\partial \Phi (D)$) possesses many
suitable tangent hyperplanes; we pause now to establish some
notation and isolate the geometric properties we will exploit.  As
noted above, we identify $\BBR^6$ with $\BBR^2 \times \BBR^2 \times
\BBR^2$; symbols such as $\vectxi$ and $\vectv$ will denote points
and vectors in $\BBR^6$, while $\xi$ and $v$ will denote points and
vectors in $\BBR^2$.  Families of points or vectors will be indexed
by superscripts, so that, for example, \[\vectxi^n = (\xi^n_1,
\xi^n_2, \xi^n_3).\]  For two quantities $A$ and $B$, we take $A
\lesssim B$ to mean $A \leq cB$ for some constant $c$; when
necessary, dependence of implied constants on certain parameters
will be denoted by subscripts on ``$\lesssim$.''

Consider a domain $\widetilde D \subset \Gamma$ as in Theorem
\ref{mainthm}, and suppose for the moment that $\partial \widetilde
D$ has curvature in a \emph{first} coordinate slice in $\BBR^6$.
Then, for some small choice of $\theta_0 >0$, we can take a
continuum $\{\vectxi^\theta \; | \; -\theta_0 \leq \theta \leq
\theta_0\}$ of points on $\partial \widetilde D$ with the following
properties:
\begin{itemize}
\item $\vectxi^\theta = (\xi_1^0, \xi_2^\theta, \xi_3^\theta)$ for
all $\theta$.
\item Let $\vectv^\theta=(v_1^\theta,v_2^\theta,v_3^\theta)$ denote a normal vector\footnote{$\partial
\widetilde D$ is a three-dimensional submanifold of the
four-dimensional subspace $\Gamma \subset \BBR^6$; we define our
normal vectors in this context, and all normal vectors are of
course chosen consistently with the orientation of $\partial
\widetilde D$.} to $\partial \widetilde D$ at the point
$\vectxi^\theta$. Elementary linear algebra shows that the
projection of $\vectv^\theta$ to the $2$-plane $\{(0,v,-v)\}
\subset \Gamma$ is
\[\frac{1}{\sqrt 2}\;\big(0, w^\theta, -w^\theta \big) := \frac{1}{\sqrt 2}\;\big(0\, ,\, v^\theta_2 - v^\theta_3\, ,\, v^\theta_3 -
v^\theta_2\big),\] and we normalize $\vectv^\theta$ so that
$|w^\theta|=1$ for all $\theta$.  By the curvature condition on
$\partial \widetilde D$, we can arrange that \[w^\theta =
\gamma_\theta w^0,\] where $\gamma_\theta \in
\operatorname{SO}_{2}(\BBR)$ denotes rotation by the angle
$\theta$.
\end{itemize}
This discussion may seem a bit cumbersome; the salient point here is
that, by the standard ``Perron tree'' construction (see
\textit{e.g.\ }\cite{stein}), the collection \[\{w^\theta =
v^\theta_2 - v^\theta_3 \: | \; -\theta_0 \leq \theta \leq
\theta_0\}\] ``yields Besicovitch sets'' in $\BBR^2$ in the
following sense:

\begin{definition}\label{besicovitch}
A family $\mathcal F$ of unit vectors in $\BBR^2$ \emph{yields
Besicovitch sets} if for every $\varepsilon >0$ there is a set
$K_\varepsilon \subset \BBR^2$ such that:
\begin{enumerate}
\item $K_\varepsilon = \bigcup_{n=1}^{N} R_n$ for some $N$ depending on $\varepsilon$, where
each $R_n$ is a rectangle of dimensions $1 \times \frac{1}{N}$.
The length-$1$ sides of each $R_n$ point in the direction of
some $v_n \in \mathcal F$.
\item $|K_\varepsilon|<\varepsilon$.
\item The rectangles $R_n'$, $1\leq n \leq N$,  are disjoint, where
$R_n'$ is obtained by translating $R_n$ by the vector $-2v_n$.
(Since each $v_n$ is a unit vector, $R_n$ and $R_n'$ are
``reaches'' of one another, in the terminology of \cite{stein}.)
\item There is a fixed compact set $K^*$ independent of
$\varepsilon$ such that $K_\varepsilon \subset K^*$.
\end{enumerate}
\end{definition}

In general, if $\widetilde D$ is as in Theorem \ref{mainthm} with
curvature in a $j_0$-th coordinate slice, then $\widetilde D$ enjoys
the following property:
\begin{property} \label{besicovitchproperty}
Given $\varepsilon >0$, there is a Besicovitch set $K_\varepsilon =
\bigcup_{1}^{N}R_n \subset K^*$ as above and a sequence of points
$\vectxi^1, \ldots, \vectxi^N \in \partial \widetilde D$ such that:
\begin{itemize}
\item $\vectxi^n = (\xi^n_1, \xi^n_2, \xi^n_3)$ with $\xi^n_{j_0} =
\xi_0$ for all $n$.
\item For all $n$, there is a normal vector $\vectv^n = (v^n_1,v^n_2,v^n_3)$ to
$\partial \widetilde D$ at $\vectxi^n$ such that the length-1
side of the rectangle $R_n$ is parallel to the vector
\[w^n_{j_0}:= v^n_{\sigma(j_0)} - v^n_{\sigma^2(j_0)} \in
\BBR^2,\] where $\sigma$ is the cycle $(1\;2\;3)$ in the
permutation group $S_3$, and $|w^n_{j_0}| =1$.
\item The vectors $\vectv^n$ all lie in a compact set $A^*$, which
is independent of $\varepsilon$.
\end{itemize}
\end{property}

\smallskip
Finally, for $\vectv \in \Gamma$, consider a half-space $P_{\vectv}
= \{\vectxi \in \Gamma \; | \; \vectxi \cdot \vectv >0\}$ with
associated trilinear form
\[\Lambda_{P_{\vectv}}(f_1,f_2,f_3) = \iiint \delta(\xi_1 + \xi_2 +
\xi_3) \chi_{P_{\vectv}}(\vectxi) \prod_j\widehat{f_j}(\xi_j)\,
\mathrm d\vectxi\] as above.  It is a matter of routine to show that the
trilinear form
\[\widetilde{\Lambda}_{\vectv}(f_1,f_2,f_3) := \int_\BBR \int_{\BBR^2} \prod_j
f_j(x-tv_j) \, \mathrm dx \, \frac{\mathrm dt}{t}\] is a linear
combination of $\Lambda_{P_{\vectv}}$ and the pointwise-product
trilinear form; here the integral in $t$ is taken in the principal
value sense.  The forms $\widetilde \Lambda_\vectv$ are of course
parameterized by the vectors $\vectv \in \Gamma$; it will also prove
useful to view them as parameterized by the triangles (or similarity
classes of triangles) in $\BBR^2$ determined by the vectors $v_j -
v_{\sigma (j)}$ (see Figure \ref{configuration}).
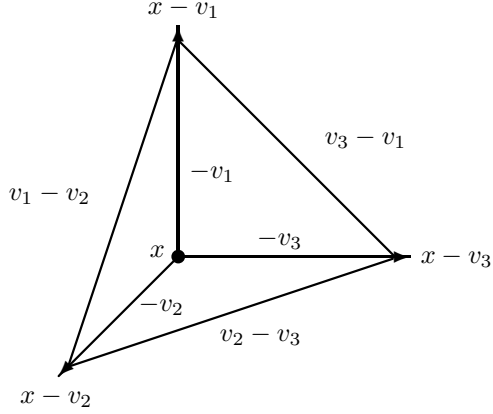
\begin{figure}
\setlength{\unitlength}{15mm}
\begin{picture}(4,4)(-1.7,-1.5)
\thicklines
\put(0,0){\circle*{.1}}
\put(0,0){\vector(0,1){2.05}}
\put(0,0){\vector(-1,-1){1.06}}
\put(0,0){\vector(1,0){2.055}}
\put(-.01,1.92){\line(-1,-3){.97}}
\put(-.98,-.98){\line(3,1){2.885}}
\put(0,1.92){\line(1,-1){1.92}}
\put(-.25,-.0){$x$}
\put(-.35,-.48){$-v_2$}
\put(.1,.7){$-v_1$}
\put(.7,.1){$-v_3$}
\put(-1.4,-1.3){$x-v_2$}
\put(2.15,-.06){$x-v_3$}
\put(-.27,2.15){$x-v_1$}
\put(-1.5,.5){$v_1-v_2$}
\put(.37,-.75){$v_2-v_3$}
\put(1.3,1){$v_3-v_1$}
\end{picture}
\caption{``Configuration triangle''}\label{configuration}
\end{figure}
The bilinear operator associated to the form $\Lambda_{P_{\vectv}}$
is the half-space Fourier multiplier $\T_{P_{\vectw}}$ on $\mathcal
S(\BBR^2) \times \mathcal S (\BBR^2)$ mentioned in the Introduction,
where $\vectw = \Phi^*(\vectv) \in \BBR^4$. Similarly, the bilinear
operator associated to $\widetilde{\Lambda}_{\vectv}$ is given by
\[\operatorname S_{\vectw}(f_1,f_2)(x) := \operatorname{p.v.} \int_\BBR f_1(x-tw_1)
f_2(x-tw_2) \, \frac{\mathrm dt}{t};\] as mentioned above, $\operatorname
S_{\vectw}$ may be viewed as a two-dimensional variant of a bilinear
Hilbert transform.

\section{Square function estimates}
The results of this section are direct analogues of the lemma of Y.
Meyer used in Fefferman's proof (Lemma 1 of \cite{fefferman}), and
their proofs are essentially identical to that of the latter.  We
begin with a domain $\widetilde D \subset \Gamma$ and a sequence of
points $\vectxi^n \in \partial \widetilde D$ at which $\partial
\widetilde D$ has normal vectors $\vectv^n$.  Let $\Lambda_n$ denote
the trilinear form $\Lambda_{P_{\vectv^n}}$ associated to the
half-space $P_n:=P_{\vectv^n}$, and set $\widetilde \Lambda_n :=
\widetilde{\Lambda}_{\vectv^n}$ as in Section \ref{notations} above.
As usual, the main idea is that, by the translation- and
dilation-invariance of bilinear multiplier norms, boundedness of the
trilinear form $\Lambda_{\widetilde D}$ will yield strong uniform
bounds (in fact $\ell^2$ vector-valued bounds) for the forms
$\Lambda_n$ or $\widetilde \Lambda_n$.

\begin{lemma}\label{meyer}
Let $0<p_1,p_2,p_3<\infty$ with $\frac{1}{p_1} + \frac{1}{p_2} +
\frac{1}{p_3} = 1$, and suppose
\[|\Lambda_{\widetilde D}(f_1,f_2,f_2)| \lesssim \prod_{j=1}^3
\|f_j\|_{p_j}\] for all measurable functions $f_1,f_2,f_3$ on
$\BBR^2$.  Let $\vectxi^n$ and $\widetilde \Lambda_n$ be as given
above.  Then:
\begin{enumerate}
\renewcommand{\labelenumi}{(\alph{enumi})}
\item\footnote{Part (a) of this lemma was originally proved in
\cite{diestelgrafakos} and \cite{grafakosreguera}.  We do not
use it in the proof of the Main Theorem; however, see the
discussion after the proof in Section \ref{mainproof}.} For all
sequences of measurable functions $f_1^n, f_2^n, f_3^n$ on
$\BBR^2$ we have the estimate
\begin{equation}\label{meyer1}
\sum_n |\widetilde \Lambda_{n}(f_1^n,f_2^n,f_3^n)|
\lesssim_{p_1,p_2,p_3} \prod_{j=1}^3 \Big\| \Big( \sum_n |f_j^n|^2
\Big)^{1/2} \Big\|_{p_j}.
\end{equation}
\item Suppose further that for some $j_0 \in \{1,2,3\}$ we have
$\xi^n_{j_0} = \xi_0$ for all $n$, and let $f_1^n, f_2^n,$ and
$f_3^n$ be sequences of measurable functions such that
$f_{j_0}^n = f_0$ for all $n$.  Then we have the estimate
\begin{equation}\label{meyer2}
\sum_n |\widetilde \Lambda_{n}(f_1^n,f_2^n,f_3^n)|
\lesssim_{p_1,p_2,p_3} \|f_0\|_{p_{j_0}} \, \prod_{j\neq j_0} \Big\|
\Big( \sum_n |f_j^n|^2 \Big)^{1/2} \Big\|_{p_j}.
\end{equation}
\end{enumerate}
\end{lemma}

The point of part (b) of this lemma is that we can afford to replace
one of the sequences $(f_j^n)_n \in L^{p_j}(\BBR^2, \ell^2)$ from
part (a) with a single function $f_0 \in L^{p_j}(\BBR^2, \BBC)$, at
the expense of requiring an extra condition on the boundary points
$\vectxi^n$; this burden will account for the coordinate-slice
restriction in the curvature conditions of the Main Theorem.
\begin{proof}  For $r>0$, let $\widetilde D_r$ denote the $r$-dilate $\{r\vectxi \;
| \; \vectxi \in \widetilde D\}$ of $\widetilde D$, and set
$\widetilde D_{r,n} = \Dr - r\vectxi^n$, so that \[\chi_{\Drn}
\longrightarrow \chi_{P_n}\] pointwise almost everywhere on $\Gamma$
as $r\rightarrow \infty$.  Then by dominated convergence we have
\begin{align*}
\Lambda_{n}(f_1,f_2,f_3) &= \lim_{r\rightarrow \infty}
\Lambda{\Drn}(f_1,f_2,f_3) \nonumber \\
& = \lim_{r\rightarrow \infty} \Lambda_{\Dr}(\M_{r\xi^n_1} f_1,
\M_{r\xi^n_2} f_2, \M_{r\xi^n_3} f_3)
\end{align*}
for all Schwartz functions $f_1,f_2$ and $f_3$, where $\M_\xi$
denotes the modulation operator defined by $\M_\xi f(x) = e^{2\pi i
\xi \cdot x}f(x)$.

Now, since we assume boundedness of the trilinear form
$\Lambda_{\widetilde D}$, the forms $\Lambda_{\Dr}$ are
\emph{uniformly} bounded on $L^{p_1} \times L^{p_2} \times L^{p_3}$,
due to the dilation-invariance of Fourier multiplier operator norms.
Summing in $n$ and appealing to Theorem 6 of
\cite{grafakosmartell}\footnote{This theorem states that if
$\operatorname T:L^{p_1}(X) \times L^{p_2}(X) \rightarrow
L^{p_3'}(X)$ is bounded, then the natural vector-valued extension of
$\T$ maps $L^{p_1}\big(X, \ell^2(\BBN)\big) \times L^{p_2}\big(X,
\ell^2(\BBN)\big)$ to $L^{p_3'}\big(X, \ell^2(\BBN \times
\BBN)\big)$ continuously.  This is a natural analogue of the
classical square function estimate for linear operators used in
\cite{fefferman}; like its linear predecessor, its proof is based on
Khintchine's inequality.} (together with a straightforward
application of duality), we obtain
\begin{align*}
\sum_n |\Lambda_{n}(f_1^n,f_2^n,f_3^n)| &= \lim_{r\rightarrow
\infty} \sum_n |\Lambda_{\Dr}(\M_{r\xi^n_1} f_1^n,
\M_{r\xi^n_2} f_2^n, \M_{r\xi^n_3} f_3^n)| \\
&\lesssim_{p_1,p_2,p_3} \lim_{r\rightarrow \infty} \prod_{j=1}^3
\Big\| \Big( \sum_k | \M_{r\xi^n_j} f_j^n|^2 \Big)^{1/2}
\Big\|_{p_j}
\\
&= \prod_{j=1}^3 \Big\| \Big( \sum_n |f_j^n|^2 \Big)^{1/2}
\Big\|_{p_j}
\end{align*} for all sequences $f_1^n, f_2^n, f_3^n \in \mathcal
S(\BBR^2)$; this is the estimate used in \cite{diestelgrafakos} and
\cite{grafakosreguera}.  Part (a) of the lemma follows immediately,
since $\widetilde \Lambda_n$ is a linear combination of $\Lambda_n$
and the pointwise-product trilinear form.

We now turn to part (b).  Fix $j_0 \in \{1,2,3\}$ such that the
points $\vectxi^n \in \partial \widetilde D$ above satisfy
$\xi^n_{j_0} = \xi_0$ for all $n$, and consider three sequences
$f_1^n, f_2^n, f_3^n$ as before, with the additional caveat that
$f_{j_0}^n = f_0$ for all $n$.  Then for each $r$, one of the
arguments of \[\Lambda_{\Dr}(\M_{r\xi^n_1} f_1^n, \M_{r\xi^n_2}
f_2^n, \M_{r\xi^n_3} f_3^n)\] is constant in $n$, so we may view
this expression as a \emph{bilinear} form in the other two
arguments.  Appealing to square function estimates for \emph{linear}
operators\footnote{Cf.\ the previous footnote.} in lieu of the
Grafakos--Martell estimate used above, we may proceed as before to
obtain
\begin{equation*}
\sum_n |\widetilde \Lambda_{n}(f_1^n,f_2^n,f_3^n)|
\lesssim_{p_1,p_2,p_3} \|f_0\|_{p_{j_0}} \, \prod_{j\neq j_0} \Big\|
\Big( \sum_n |f_j^n|^2 \Big)^{1/2} \Big\|_{p_j}
\end{equation*} as desired; the term $\|f_0\|_{p_{j_0}}$ appears via
the norm of the linear operator associated to the aforementioned
bilinear form.
\end{proof}

\section{Proof of the Main Theorem}\label{mainproof}
Let $\widetilde D$ be a domain in $\Gamma$ with curvature in a
$j_0$-th coordinate slice, and suppose toward a contradiction that
we have the estimate
\[|\Lambda_{\widetilde D}(f_1,f_2,f_3)| \lesssim \prod_{j=1}^3
\|f_j\|_{p_j}.\]  The goal now is to establish sufficient lower
bounds for $\widetilde \Lambda_n(f_1^n, f_2^n, f_3^n)$ to contradict
part (b) of Lemma \ref{meyer}; of course, this entails a suitable
choice of the points $\vectxi^n \in
\partial \widetilde D$ that define $\widetilde \Lambda_n$, as well as a
suitable choice of the functions $f_j^n$. As we resign ourselves to
following Fefferman's approach, we will eventually take the $f_j^n$
to be characteristic functions of aptly chosen rectangles in
$\BBR^2$, and the lower bounds on $\widetilde \Lambda_n$ will just
be pointwise lower bounds for the operators $\T_{P_n}$ in disguise.

Recall that for the moment we work within the Banach triangle
outside the local $L^2$ setting; namely, we seek to prove the
unboundedness of $\Lambda_{\widetilde D}$ on $L^{p_1}(\BBR^2) \times
L^{p_2}(\BBR^2) \times L^{p_3}(\BBR^2)$, with $\sum \frac{1}{p_{j}}
=1$ and $p_{i} <2$ for exactly one $i \neq j_0$.\footnote{For the
remainder of the proof we suppress all dependence of constants on
the exponents $p_j$.}  The key insight of Fefferman is that one can
exploit Besicovitch sets to make the right-hand side of the square
function estimate (\ref{meyer2}) arbitrarily small while keeping the
left-hand side large; achieving this is only slightly more involved
in our setting than in the linear. Indeed, let $K_\varepsilon$ be a
Besicovitch set as in Definition \ref{besicovitch}, with
$K_\varepsilon = \bigcup_1^N R_n$. Then H\"older's inequality yields
\begin{equation}\label{pless2}
\Big \| \Big(\sum_n |\chi_{R_n}|^2 \Big)^{1/2}\Big\|_{p_i} \, < \,
\varepsilon^{\frac{2-p_i}{2p_i}},
\end{equation}
which can be made arbitrarily small by decreasing $\varepsilon$,
since $p_i<2$.  To exploit this estimate, we use the fact that
$\partial \widetilde D$ has curvature in a $j_0$-th coordinate
slice; taking any $\varepsilon >0$, we have points $\vectxi^1,
\ldots ,\vectxi^N$, normal vectors $\vectv^1, \ldots, \vectv^N$, and
a Besicovitch set $K_\varepsilon = \bigcup_1^N R_n$ provided by
Property \ref{besicovitchproperty} from Section \ref{notations}. In
light of estimate (\ref{pless2}) we set $f_i^n = \chi_{R_n}$, so
that part (b) of Lemma \ref{meyer} gives
\begin{equation}\label{estimate1}
\sum_{n=1}^N |\widetilde \Lambda_n(f^n_1,f^n_2,f^n_3)| \;\lesssim\;
\varepsilon^{\frac{2-p_i}{2p_i}} \;\|f_{0}\|_{p_{j_0}} \;\Big\|
\Big( \sum_{n=1}^N |f_k^n|^2 \Big)^{1/2}\Big\|_{p_k}
\end{equation}
for all sequences of functions $f_k^n$ and $f_{j_0}^n$ such that
$f_{j_0}^n = f_0$ for all $n$; here $k \in \{1,2,3\}$ is the
remaining index $j_0 \neq k \neq i$.

\bigskip We now produce $f_0$ and $f_k^n$ that will essentially
maximize the left-hand side of (\ref{estimate1}).  In fact, we will
set $f_k^n = \chi_{R_n'}$, where $R_n'$ is the reach\footnote{Of
course, there are two choices of reach depending on our choice of
orientation; we will always choose the one that is obviously
expedient.} of $R_n$ as given in Definition \ref{besicovitch}, and
$f_0 = \chi_Q$ for some large rectangle $Q$ to be determined. In
this setting we have
\begin{align}
\widetilde \Lambda_n (f_1^n,f_2^n,f_3^n) & = \int_\BBR \int_{\BBR^2}
\prod_{j=1}^3\, f_j^n (x - tv_j^n) \; \mathrm dx \, \frac{\mathrm dt}
{t} \nonumber\\
&= \iint \chi_{R_n}(x-tv_i^n) \,\chi_{R_n'}(x - tv_k^n) \,\chi_Q (x
- v_{j_0}^n) \; \mathrm dx \,\frac{\mathrm dt}{t} \nonumber\\
& = \int_\BBR \int_{R_n} \chi_{R_n'}\big( x - t(v^n_k - v^n_i) \big)
\; \chi_Q\big( x - t(v^n_{j_0} - v^n_i) \big) \; \mathrm dx \,
\frac{\mathrm dt}{t}.
\label{sliding}
\end{align}
Note here that the vector $v^n_k - v^n_i$ is equal to
$\pm(v^n_{\sigma(j_0)} - v^n_{\sigma^2(j_0)})$, so that $R_n$ is
parallel to this vector with $R_n' = R_n - 2(v^n_k - v^n_i)$ by
Property \ref{besicovitchproperty}.  Also note that $\vectv^n$ is
chosen by design so that $|v^n_k - v^n_i| = 1$; here and in what
follows, it may be helpful to think in terms of the configuration
triangles of Section \ref{notations} (see Figure
\ref{slidingfigure}).
\begin{figure}
\setlength{\unitlength}{12mm}
\begin{picture}(8,6.5)(-3,-2.3)
\thicklines
\put(-.01,1.92){\vector(-1,-3){.632}}
\put(1.91,0){\vector(-1,0){2.56}}
\put(1.92,0){\vector(-1,1){1.92}}
\put(-1.15,.9){$\scriptstyle{v_k^n-v_i^n}$}
\put(.37,-.3){$\scriptstyle{v_{j_0}^n-v_i^n}$}
\put(1.2,1){$\scriptstyle{v_{j_0}^n-v_k^n}$}
\put(-.4,2.05){\line(1,3){.632}}
\put(-.4,2.05){\line(-3,1){.5}}
\put(-.9,2.22){\line(1,3){.632}}
\put(.23,3.95){\line(-3,1){.5}}
\put(-1.04,.08){\line(-1,-3){.632}}
\put(-1.04,.08){\line(-3,1){.5}}
\put(-1.54,.25){\line(-1,-3){.632}}
\put(-1.67,-1.82){\line(-3,1){.5}}
\put(2.2,.4){\line(1,0){2.4}}
\put(4.6,.4){\line(0,-1){2.7}}
\put(2.2,.4){\line(0,-1){2.7}}
\put(2.2,-2.3){\line(1,0){2.4}}
\put(3.3,-1){$Q$}
\put(-1.8,-.9){$R_n$}
\put(-.525,3){$R_n'$}
\end{picture}
\caption{}\label{slidingfigure}
\end{figure}
 Given these observations, it is clear from
(\ref{sliding}) that if we choose $Q$ to be a large enough rectangle
we have
\begin{equation}\label{lowerbound}
|\widetilde \Lambda_n (f_1^n,f_2^n,f_3^n)| \;\gtrsim \;|R_n| \;=\;
\frac{1}{N}.
\end{equation}
Moreover, recall that we have the provisos $K_\varepsilon \subset
K^*$ for all $\varepsilon$ and some fixed compact set $K^*$, and
$\vectv^n \in A^*$ for some compact $A^*$ which is again independent
of $\varepsilon$.  Thus we may in fact choose such a $Q$
\emph{independently} of $n$ and $\varepsilon$; the important fact
here is that we can set $f_{j_0}^n = f_0 := \chi_Q$ for all $n$ and
ensure \[\|f_0\|_{p_{j_0}} \lesssim 1\] independently of
$\varepsilon$.  If the preceding discussion seems lacking in
motivation, the reader may note that in the heuristic limiting case
of $Q= \BBR^2$, the bilinear forms \[\Lambda'_n(f,g) := \widetilde
\Lambda_n(f,g,\chi_{\BBR^2})\] correspond to the directional Hilbert
transforms (or linear half-space multipliers) used in Fefferman's
original argument.  (Note also that in the case $p_{j_0}=\infty$
this observation can be used trivially to deduce unboundedness from
Fefferman's proof.) Finally, we observe that since the rectangles
$R_n'$ are \emph{disjoint}, we also have
\[\Big\| \Big( \sum_{n=1}^N |f_k^n|^2 \Big)^{1/2}\Big\|_{p_k} \; =
\; \Big\| \sum_{n=1}^N \chi_{R_n'} \Big\|_{p_k} = 1.\]

\smallskip
Thus, combining estimates (\ref{estimate1}) and (\ref{lowerbound}),
we obtain \[1 \lesssim \sum_{n=1}^N |\widetilde \Lambda_n
(f_1^n,f_2^n,f_3^n)| \lesssim \varepsilon^{\frac{2-p_i}{2p_i}};\]
this renders our original boundedness assumption absurd and
completes the proof of the Main Theorem. \hfill $\Box$

\bigskip
For the sake of contrast, we now briefly summarize the approach of
\cite{diestelgrafakos} and \cite{grafakosreguera}, which actually
yields the following:

\begin{theorem}\label{grafakostheorem}
Let $\widetilde D$ be a domain in $\Gamma$ such that $\partial
\widetilde D$ possesses a family of normal vectors
\[\big\{\vectv^\theta = \big(v_\theta, \lambda_\theta v_\theta,
-(1+\lambda_\theta)v_\theta\big) \; | \; \theta \in \mathcal
I\big\}\subset \Gamma\] for some index set $\mathcal I$, where the
collection $\{v_\theta\, |\, \theta \in \mathcal I\} \subset \BBR^2$
yields Besicovitch sets as in Definition \ref{besicovitch}, and
$\lambda_\theta \sim 1$ for all $\theta \in \mathcal I$. Then the
trilinear form $\Lambda_{\widetilde D}$ is not of type $\vectp =
(p_1,p_2,p_3)$ whenever $\sum \frac{1}{p_j} =1$ and $p_i < 2$ for
some $i \in \{1,2,3\}$.
\end{theorem}
This approach also follows Fefferman's argument, using part (a) of
Lemma \ref{meyer} where we have used part (b).  Note that appealing
to part (a) allows one to eliminate any restriction on the points of
$\partial \widetilde D$ at which the normal vectors $\vectv^\theta$
occur; in exchange, use of the square function estimate
(\ref{meyer1}) forces a rather stringent condition on the normal
vectors themselves (\emph{viz.}, all three of the component vectors
of $\vectv^\theta$ must be parallel in $\BBR^2$).  To prove the
theorem, as above one chooses appropriate normal vectors $\vectv^n =
\vectv^{\theta_n}$ associated to some Besicovitch set $K_\varepsilon
= \bigcup_1^N R_n$ and considers the forms $\widetilde \Lambda_n :=
\widetilde \Lambda_{\vectv^n}$.  Again setting $f_i^n = \chi_{R_n}$,
part (a) of Lemma \ref{meyer} then gives
\begin{equation} \label{estimate2}
\sum_{n=1}^N | \widetilde \Lambda_n (f_1^n, f_2^n, f_3^n)|\;
\lesssim \; \varepsilon^{\frac{2-p_i}{2p_i}}\; \prod_{j \neq i}
\Big\| \Big( \sum_{n=1}^N |f_j^n|^2 \Big)^{1/2}\Big\|_{p_j}.
\end{equation}
At this point, if one wishes to follow Fefferman by setting $f_j^n =
\chi_{Q_j^n}$ for some rectangles $Q_j^n$, any productive use of
estimate (\ref{estimate2}) clearly prohibits one from taking
$|Q_j^n| \gtrsim 1$ given that $p_j > 2$ for $j \neq i$.  However,
since the component vectors $v_j^n$ of $\vectv^n$ are all parallel,
the configuration triangle for $\widetilde \Lambda_n$ is degenerate
(\textit{i.e}.\ all vertices are collinear); thus one may choose
$Q_j^n$ to be appropriate ``reaches'' of $R_n$ for \emph{both}
$j\neq i$ and still obtain \[|\widetilde \Lambda_n (f_1^n, f_2^n,
f_3^n)| \gtrsim \frac{1}{N}\] as above.  See Figure
\ref{degeneratefigure}, which should be contrasted with Figure
\ref{slidingfigure}. Since we were able to choose $f_j^n =
\chi_{Q_j^n}$ with $|Q_j^n| \sim \frac{1}{N}$ and $\{Q_j^n\}_n$
disjoint for each $j \neq i$, the right-hand side of
(\ref{estimate2}) is controlled by $\varepsilon^{\frac{2-
p_i}{2p_i}}$, and we obtain a contradiction as above.
\begin{figure}
\setlength{\unitlength}{1cm}
\begin{picture}(10,3)(-6.5,-2)
\thicklines
\put(-4,-1){\circle*{.1}}
\put(-4,-1){\vector(1,0){6}}
\put(-4,-1){\vector(1,0){4.5}}
\put(2,-.5){\line(1,0){1.5}}
\put(2,-.5){\line(0,1){.5}}
\put(2,0){\line(1,0){1.5}}
\put(3.5,-.5){\line(0,1){.5}}
\put(2.6,.2){$R_n$}
\put(-1,-.5){\line(1,0){1.5}}
\put(-1,-.5){\line(0,1){.5}}
\put(-1,0){\line(1,0){1.5}}
\put(.5,-.5){\line(0,1){.5}}
\put(-.9,.2){$Q_1^n = R_n'$}
\put(-6.5,0){\line(1,0){2.5}}
\put(-6.5,-.5){\line(0,1){.5}}
\put(-6.5,-.5){\line(1,0){2.5}}
\put(-4,-.5){\line(0,1){.5}}
\put(-5.5,.2){$Q_2^n$}
\end{picture}
\caption{Degenerate configuration triangle}\label{degeneratefigure}
\end{figure}
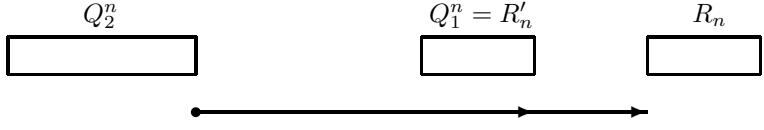

In fact, an examination of the geometric considerations in the proof
shows that an approach toward contradicting part (a) of Lemma
\ref{meyer} essentially necessitates the use of degenerate
configuration triangles, provided one insists on exploiting
Besicovitch sets and taking the $f_j^n$ to be characteristic
functions of rectangles.  Thus, with such an approach one cannot
dispense with the restriction on the normal vectors appearing in
Theorem \ref{grafakostheorem}; in particular, one cannot treat
generic strictly convex domains in $\Gamma$.

\section{Unboundedness on the border of and outside the Banach triangle}\label{nonbanach}
Our arguments thus far have been restricted to exponent triples
$\vectp = (p_1,p_2,p_3)$ in the interior of the ``Banach triangle,''
with $1 < p_1,p_2,p_3 < \infty$; as mentioned in the Introduction,
one can also obtain unboundedness results for $\vectp$ outside this
range.  However, to phrase such a general result in terms of
trilinear forms, the notion of ``type $\vectp$'' is unsuitable; in
fact, it is not hard to see that for $\vectp$ outside the Banach
triangle the \emph{only} trilinear form of type $\vectp$ is the $0$
form (see \textit{e.g.\ }Chapter 3 of \cite{thiele}).  On the other
hand, it is perfectly reasonable (\textit{i.e}.\ nontrivial) to
investigate the boundedness of bilinear operators from $L^p \times
L^q$ to $L^r$ with $\frac{1}{2} \leq r \leq 1$ (so that $r' \leq -1$
or $r' = \infty$), and one would hope to be able to treat such
questions symmetrically. One way of dealing with this state of
affairs is to replace the notion of type $\vectp$ with that of
``generalized restricted type $\vectp$,'' which we review below
(once again, the reader may consult \cite{thiele} for a more
detailed treatment of this formalism).\footnote{Our notation is
``reciprocal'' to that of \cite{thiele}; in the notation therein,
our ``generalized restricted type $\vectp$'' corresponds to
generalized restricted type $\alpha = \Big(\frac{1}{p_1},
\frac{1}{p_2}, \frac{1}{p_3}\Big)$.} Nonetheless, the reader will
notice that we are forced to abandon the symmetric framework of
trilinear forms when treating the boundary of the Banach triangle,
where $p_j=\infty$ for some $j$.

Let us call an exponent triple $\vectp = (p_1,p_2,p_3)$
\emph{admissible} if $|p_k| \geq 1$ for all $k$, $p_j \leq -1$ for
at most one $j$, and $\sum_{k=1}^3 p_k^{-1} = 1$; the $p_k$ are of
course allowed to be infinite.  For such $\vectp$, we say the form
$\Lambda$ is of \emph{generalized restricted type $\vectp$} if for
all triples $(E_1, E_2, E_3)$ of measurable subsets of $\BBR^2$
there exists a subset $\widetilde E_j \subset E_j$ with $|\widetilde
E_j| \geq \frac{1}{2}|E_j|$ for which we have the estimate
\[|\Lambda(f_1,f_2,f_3)| \lesssim \prod_{k=1}^3 |E_k|^{1/p_k}\]
whenever $|f_k| \leq \chi_{E_k}$ and moreover $|f_j| \leq
\chi_{\widetilde E_j}$.  (If in fact $p_k \geq 1$ for all $k$, then
the exceptional index $j$ may be chosen at will.)  Of course, inside
the Banach triangle generalized restricted type $\vectp$ is implied
by type $\vectp$; a generalized restricted type estimate for a form
$\Lambda$ gives a restricted weak-type estimate for an appropriate
bilinear dual of the operator associated to $\Lambda$.

Finally, we introduce some geometric terminology.  We say that a
vector $\vectv \in \Gamma$ is \emph{degenerate} if $\vectv =
(\lambda_1 v, \lambda_2 v, \lambda_3 v)$ for some $v \in S^1$ and
some scalars $\lambda_j \in \BBR$. Furthermore, we say that $\vectv$
is \emph{strongly degenerate} if $\lambda_j = 0$ for some $j$.  A
domain $\widetilde D \subset \Gamma$ is called (\emph{strongly})
\emph{degenerate} if \emph{every} normal vector to $\partial
\widetilde D$ is (strongly) degenerate.  In the following discussion
we will omit any consideration of strongly degenerate domains; if
$\partial \widetilde D$ is smooth such domains are given by
particular cylinder sets, and the boundedness properties of their
associated forms (or operators) fall within the purview of the
linear theory.

The following theorem\footnote{It has been pointed out to the author
by C. Thiele and C. Demeter that this result has been established
independently as folklore, with essentially the same proof; the
author is also indebted to C. Thiele for suggesting the argument
used to prove part (b) of the theorem.} shows that, in sufficiently
nondegenerate cases, the boundary curvature of a domain is actually
\emph{irrelevant} to the boundedness properties of the associated
trilinear form or bilinear operator.

\begin{theorem}\label{halfplanetheorem}
Let $P_{\vectv} = \{\vectxi \in \Gamma \, | \, \vectxi \cdot \vectv
>0\}$ be a nondegenerate half-space in $\Gamma$, and let $\vectp =
(p_1,p_2,p_3)$ be an admissible triple.
\begin{enumerate}\renewcommand{\labelenumi}{(\alph{enumi})}
\item If $p_i \leq -1$ for some $i$, the trilinear form
$\Lambda_{P_\vectv}$ is not of generalized restricted type
$\vectp$.
\item If $p_i = \infty$ for some $i$, the bilinear multiplier
operator $\T_{\Phi^{-1}(P_\vectv)}$ associated to the canonical
preimage of $P_\vectv$ in $\BBR^4$ is unbounded from
$L^{p_1}(\BBR^2) \times L^{p_2}(\BBR^2)$ to $L^{p_3'}(\BBR^2)$.
\end{enumerate}
\end{theorem}

\begin{proof}
As usual, we will prove the equivalent statements for the trilinear
form $\widetilde{\Lambda}_\vectv$ or the bilinear operator
$\operatorname S_{\vectw}$ as in Section \ref{notations}.  To prove
part (a), we proceed as in Figure \ref{halfplanefigure}. For $j \neq
i \neq k$, set $f_j$ to be the characteristic function of a
rectangle $R$ of width $\varepsilon$ and length $1$ oriented
parallel to $v_j - v_k$, and let $f_k$ be the characteristic
function of its reach $R'$ (as above, we normalize $\vectv$ so that
$|v_j - v_k|=1$).  Let $f_i$ be the characteristic function of a
cube $Q$ contained in the intersection of the strips $R + \BBR \cdot
(v_j - v_i)$ and $R' + \BBR \cdot (v_k - v_i)$; choose $Q$ to have
measure comparable to $1$.  Computation yields
\[|\widetilde{\Lambda}_\vectv(f_1,f_2,f_3)| \gtrsim_{\vectv} \varepsilon,\]
while
\[|R|^{1/p_j}\,|R'|^{1/p_k}\,|Q|^{1/p_i} \sim
\varepsilon^{\frac{1}{p_j} + \frac{1}{p_k}} =
\varepsilon^{1/p_i'}.\]  Both estimates continue to hold after the
excision of any half-measure set from the cube $Q$, and sending
$\varepsilon$ to zero violates the generalized restricted type
estimate since $\frac{1}{p_i'} >1$.

\begin{figure}
\setlength{\unitlength}{12mm}
\begin{picture}(8,5)(-4,-4)
\thicklines
\multiput(-3,-3)(2,2){2}{
    \put(0,0){\line(1,1){1}}
    \put(0,0){\line(1,-1){.25}}
    \put(.25,-.25){\line(1,1){1}}
    \put(1,1){\line(1,-1){.25}}
    }
\thinlines
\multiput(-4,-3.25)(.4,0){15}{
    \line(1,0){.2}}
\multiput(-4,-2)(.4,0){15}{
    \line(1,0){.2}}
\multiput(-1.125,0)(0,-.4){11}{
    \line(0,1){.2}}
\multiput(.18,0)(0,-.4){11}{
    \line(0,1){.2}}
\thicklines
\put(-.7,-2.4){\line(1,0){.6}}
\put(-.7,-2.4){\line(0,-1){.6}}
\put(-.1,-3){\line(0,1){.6}}
\put(-.1,-3){\line(-1,0){.6}}
\put(-.5,-2.75){$Q$}
\put(-2.5,-2.7){$R$}
\put(-.5,-.7){$R'$}
\put(-.85,-3.5){\vector(1,0){1}}
\put(-.8,-3.75){$v_j-v_i$}
\put(.6,-2.15){\vector(0,-1){1}}
\put(.725,-2.75){$v_k-v_i$}
\put(-2.25,-1.8){\vector(1,1){1}}
\put(-2.7,-1.1){$v_j-v_k$}
\end{picture}
\caption{}
\label{halfplanefigure}
\end{figure}

This counterexample can be slightly modified to prove part (b).
Since the nondegeneracy condition on $P_\vectv$ is preserved under
duality (\emph{i.e.}, permutation of the three coordinates in
$\BBR^2 \times \BBR^2 \times \BBR^2$), we may assume that $p_3 =
\infty$; thus it suffices to prove that $\operatorname S_{\vectw}$
is unbounded from $L^p \times L^{p'}$ to $L^1$, where $\vectw =
\Phi^*(\vectv) = (v_1 - v_3\, , \, v_2 - v_3)$.  As in the
counterexample above, we take $f_1 = \chi_R$, where $R$ is a
rectangle of length $1$ and width $\varepsilon$ oriented parallel to
$w_1 - w_2 = v_1 - v_2$; however, instead of obtaining $f_2$ from
the reach of $R$, we simply set $f_2 = f_1 = \chi_R$.  Computation
then yields \[\| \operatorname S_\vectw (f_1,f_2)\|_1
\gtrsim_{\vectw} -\varepsilon \log \varepsilon,\] while $\|f_1\|_p
\, \|f_2\|_{p'} = \varepsilon$, and again sending $\varepsilon$ to
zero yields unboundedness.
\end{proof}

One should note that the logarithmic divergence in the proof of part
(b) occurs because $\operatorname S(\chi_R, \chi_R)$ is large on a
region approaching the boundary of the rectangle $R$; thus if we
study the expression \[\widetilde \Lambda_\vectv (\chi_R, \chi_R,
\chi_E) = \int_{\BBR^2} \operatorname S_\vectw (\chi_R, \chi_R) \,
\chi_E\] for any measurable set $E$, we are not at liberty to delete
an \emph{arbitrary} half-measure subset of $E$ and still obtain this
divergence.  Therefore, this counterexample cannot violate
generalized restricted type $\vectp$ estimates for $\vectp$ on the
boundary of the Banach triangle; this state of affairs could be
viewed as analogous to the fact that the Hilbert transform is
unbounded on $L^1$ but is in fact of weak-type $(1,1)$ (though we
are of course making no claims of any such weak-type bounds in the
present bilinear setting).

The reader should also note that the nondegeneracy assumption is not
merely an artifact of the proof; indeed, if $P_\vectv$ is
degenerate, the operator $\operatorname S_{\Phi^*(\vectv)}$ inherits
its boundedness properties from those of a one-dimensional bilinear
Hilbert transform.

With this discussion completed, we can finally state the Main
Theorem in full generality:
\begin{maintheorem}[General version]\label{mainthmgen}
Let $\widetilde D$ be a domain in $\Gamma \subset \BBR^6$ as in the
Main Theorem above, with nontrivial curvature in a $j$-th coordinate
slice for some $j \in \{1,2,3\}$.  Assume further that $\widetilde
D$ is not strongly degenerate.  Then for admissible triples $\vectp
= (p_1,p_2,p_3)$, the trilinear form $\Lambda_{\widetilde D}$ fails
to be of generalized restricted type $\vectp$ whenever:
\begin{itemize}
\item $p_i <2$ for some $i \neq j$,
\item $p_i \leq -1$ for some $i$ (\emph{i.e.}, $\vectp$ lies outside
the Banach triangle).
\end{itemize}
If $D = \Phi^{-1}(\widetilde D)$ is the canonical preimage of
$\widetilde D$ in $\BBR^4$, then the operator $\T_D$ is unbounded
from $L^{p_1}(\BBR^2) \times L^{p_2}(\BBR^2) \rightarrow
L^{p_3'}(\BBR^2)$ whenever $p_i = \infty$ for some $i$ (\emph{i.e.},
whenever $\vectp$ lies on the border of the Banach triangle) and
additionally $p_j \neq 2$ for all $j$.  If $\widetilde D$ is further
assumed to be nondegenerate, then the restriction $p_j \neq 2$ can
be removed.
\end{maintheorem}

\begin{proof}
To treat the first case, one can observe that in the interior of the
Banach triangle generalized restricted type $\vectp$ estimates imply
imply ``restricted type $\vectp$'' estimates, and, since our
counterexamples were constructed from characteristic functions, our
proofs thus far can be applied (cf.\ Lemma 3.6 of \cite{thiele}).
Thus, we need only concern ourselves with $\vectp$ on the border of
or outside the Banach triangle.  If $\widetilde D$ is nondegenerate,
one simply combines Theorem \ref{halfplanetheorem} with the usual
dilation- and translation-invariance of bilinear multiplier norms to
obtain the desired unboundedness or failure of generalized
restricted type.

If $\widetilde D$ is degenerate, however, we need to exploit the
curvature of $\partial \widetilde D$.  Note that, since $\widetilde
D$ is assumed not to be strongly degenerate, degeneracy of
$\widetilde D$ and the coordinate-slice curvature hypothesis imply
that $\widetilde D$ must in fact satisfy the hypotheses of Theorem
\ref{grafakostheorem}.  For $\vectp$ outside the Banach triangle,
the proof of Theorem \ref{grafakostheorem} carries over after the
excision of half-measure sets wherever necessary, and one obtains
the failure of generalized restricted type $\vectp$ for
$\Lambda_{\widetilde D}$.  If $\vectp$ lies on the boundary of the
Banach triangle, by duality it suffices to disprove $L^p \times
L^{p'} \rightarrow L^1$ bounds on $\T_D$ for any $\widetilde D$
satisfying the hypotheses of Theorem \ref{grafakostheorem}; note
that these hypotheses are again symmetric under permutation of
coordinates.  If $1 \neq p \neq \infty$, this can be accomplished by
following the methods of \cite{diestelgrafakos} or
\cite{grafakosreguera} (\emph{i.e.}, the proof of Theorem
\ref{grafakostheorem} phrased in terms of bilinear operators).
Finally, if $p=1$ or $p=\infty$, we use dilations and translations
of $D$ to pass to a half-space multiplier $\T_P$; from this point,
we simply consider either $\T_P(1, \,\cdot\,)$ or $\T_P(\, \cdot \,
, 1)$ and invoke the unboundedness of the Hilbert transform on
$L^1$.
\end{proof}

In summary, the presently known range of unboundedness of $\T_D$ or
$\Lambda_{\widetilde D}$ for nondegenerate $\widetilde D \subset
\Gamma$ with boundary curvature in a first coordinate slice is given
by the shaded region of the type diagram in Figure
\ref{banachtriangle}; of course, the corresponding ranges for the
other two slice curvature conditions are given by rotations of this
diagram.  If $\widetilde D$ is merely assumed not to be strongly
degenerate, we are forced to omit the vertices of the local $L^2$
region; we reiterate that the status of the bilinear multiplier
problem in the unshaded regions is completely unknown at present.

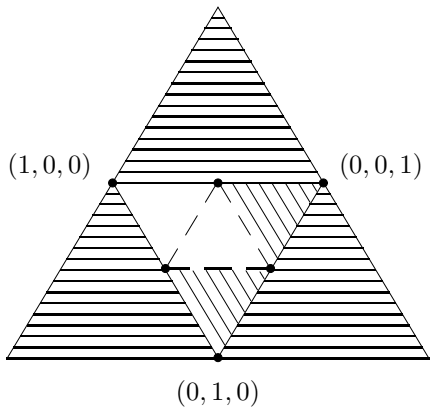
\begin{figure}
\setlength{\unitlength}{7mm}
\begin{picture}(10,9)(-5,-5)
%\thicklines
\put(-4.01,-4.02){\line(1,0){8.02}}
\put(-4,-4){\line(3,5){4}}
\put(4,-4){\line(-3,5){4}}
\put(0,-4){\line(-3,5){2}}
\put(0,-4){\line(3,5){2}}
\put(-2,-0.68){\line(1,0){4}}
\multiput(-0.02,-0.68)(-.375,-.625){3}
    {\line(-3,-5){.2222}}
\multiput(0.02,-0.68)(.375,-.625){3}
    {\line(3,-5){.2222}}
\put(-1,-2.31){\line(1,0){.45}}
\put(-.25,-2.31){\line(1,0){.5}}
\put(1,-2.31){\line(-1,0){.45}}
%SHADE TOP
\put(-1.875,-.47){\line(1,0){3.75}}
\put(-1.75,-.26){\line(1,0){3.5}}
\put(-1.625,-.05){\line(1,0){3.25}}
\put(-1.5,.16){\line(1,0){3}}
\put(-1.375,.37){\line(1,0){2.75}}
\put(-1.25,.58){\line(1,0){2.5}}
\put(-1.125,.79){\line(1,0){2.25}}
\put(-1,1){\line(1,0){2}}
\put(-.875,1.21){\line(1,0){1.75}}
\put(-.75,1.42){\line(1,0){1.5}}
\put(-.625,1.63){\line(1,0){1.25}}
\put(-.5,1.84){\line(1,0){1}}
\put(-.375,2.05){\line(1,0){.75}}
\put(-.25,2.26){\line(1,0){.5}}
\put(-.125,2.45){\line(1,0){.25}}
%SHADE BOTTOM LEFT
\put(-2,-3.34){
\put(-1.875,-.47){\line(1,0){3.75}}
\put(-1.75,-.26){\line(1,0){3.5}}
\put(-1.625,-.05){\line(1,0){3.25}}
\put(-1.5,.16){\line(1,0){3}}
\put(-1.375,.37){\line(1,0){2.75}}
\put(-1.25,.58){\line(1,0){2.5}}
\put(-1.125,.79){\line(1,0){2.25}}
\put(-1,1){\line(1,0){2}}
\put(-.875,1.21){\line(1,0){1.75}}
\put(-.75,1.42){\line(1,0){1.5}}
\put(-.625,1.63){\line(1,0){1.25}}
\put(-.5,1.84){\line(1,0){1}}
\put(-.375,2.05){\line(1,0){.75}}
\put(-.25,2.26){\line(1,0){.5}}
\put(-.125,2.45){\line(1,0){.25}}
}
%SHADE BOTTOM RIGHT
\put(2,-3.34){
\put(-1.875,-.47){\line(1,0){3.75}}
\put(-1.75,-.26){\line(1,0){3.5}}
\put(-1.625,-.05){\line(1,0){3.25}}
\put(-1.5,.16){\line(1,0){3}}
\put(-1.375,.37){\line(1,0){2.75}}
\put(-1.25,.58){\line(1,0){2.5}}
\put(-1.125,.79){\line(1,0){2.25}}
\put(-1,1){\line(1,0){2}}
\put(-.875,1.21){\line(1,0){1.75}}
\put(-.75,1.42){\line(1,0){1.5}}
\put(-.625,1.63){\line(1,0){1.25}}
\put(-.5,1.84){\line(1,0){1}}
\put(-.375,2.05){\line(1,0){.75}}
\put(-.25,2.26){\line(1,0){.5}}
\put(-.125,2.45){\line(1,0){.25}}
}
%DOTS
\put(0,-0.68){\circle*{.15}} \put(-2,-0.68){\circle*{.15}}
\put(2,-0.68){\circle*{.15}} \put(0,-4){\circle*{.15}}
\put(1,-2.31){\circle*{.15}} \put(-1,-2.31){\circle*{.15}}
%SHADE RIGHT MID
\put(1.125,-2.1){\line(-3,5){.85}}
\put(1.25,-1.89){\line(-3,5){.725}}
\put(1.375,-1.68){\line(-3,5){.6}}
\put(1.5,-1.47){\line(-3,5){.475}}
\put(1.625,-1.26){\line(-3,5){.35}}
\put(1.75,-1.05){\line(-3,5){.22}}
\put(1.875,-.84){\line(-3,5){.08}}
%SHADE BOTTOM MID
\put(-1,-1.66){
\put(1.125,-2.1){\line(-3,5){.85}}
\put(1.25,-1.89){\line(-3,5){.725}}
\put(1.375,-1.68){\line(-3,5){.6}}
\put(1.5,-1.47){\line(-3,5){.475}}
\put(1.625,-1.26){\line(-3,5){.35}}
\put(1.75,-1.05){\line(-3,5){.22}}
\put(1.875,-.84){\line(-3,5){.08}} }
%EXPONENTS
\put(-4,-.5){$(1,0,0)$}
\put(2.3,-.5){$(0,0,1)$}
\put(-.8,-4.8){$(0,1,0)$}
\end{picture}\caption{Type diagram of points $\big(\frac{1}{p_1},
\frac{1}{p_2},\frac{1}{p_3}\big)$} \label{banachtriangle}
\end{figure}

\section{Open directions and remarks}\label{remarks}
\subsection{Open directions}
\subsubsection*{More exotic domains}
Even in view of Proposition \ref{muscaluprop}, the curvature
conditions of the Main Theorem may seem somewhat \textit{ad hoc}.
There exist less degenerate domains $\widetilde D \subset \Gamma$
whose boundaries have nontrivial principal curvature at a point but
are locally flat in the three coordinate directions of $\BBR^2
\times \BBR^2 \times \BBR^2$; for an example, consider the domain
$\widetilde D_1 = \Phi(D_1)$, with
\[D_1 = \{ (\xi_1,\xi_2,\xi_3,\xi_4) \in \BBR^4 \; | \;
\xi_4 > \xi_1\xi_3 + \xi_1^2\}.\] This domain falls outside the
scope of both the techniques of this paper and those of
\cite{diestelgrafakos} and \cite{grafakosreguera}.

\subsubsection*{Comparison of the methods}
It should be noted here that while Corollary \ref{convexcor}
generalizes Theorem 1 of \cite{grafakosreguera}, the \emph{methods}
of this paper do not appear to be strictly stronger than those of
Grafakos \textit{et al.}  In short, their methods require the
availability of rather specific normal vectors, but there is no
restriction on the boundary points at which these normal vectors
occur as in our slice conditions; see Theorem \ref{grafakostheorem}.
It is easy to construct examples of domains satisfying the
hypotheses of our Main Theorem but failing those of Theorem
\ref{grafakostheorem} (cf.\ Corollary \ref{convexcor}); examples
treatable by Theorem \ref{grafakostheorem} but not the Main Theorem
seem less trivial to produce.  For instance, elementary algebraic
arguments show that one cannot find such an example $\widetilde D =
\Phi(D)$ with $\partial D$ a quadratic subvariety of $\BBR^4$;
however, the argument seems particular to the quadratic setting, and
it could perhaps be interesting to find such examples in general.

\subsubsection*{Untreated ranges of exponents}
Finally, of course there remains the question of the exact range of
$L^p$ spaces for which one should expect unboundedness results.  An
obvious problem is to consider a domain $\widetilde D$ satisfying
exactly one of the coordinate-slice curvature conditions and address
the omitted triangle lying outside local $L^2$ but within the Banach
triangle (see Figure \ref{banachtriangle}). This region seems beyond
the reach of the rather standard methods used in this paper; in
short, one needs to exploit the small area of Besicovitch sets by
measuring the appropriate square function in $L^{p_j}$ with $p_j
<2$, but if one only has curvature in a $j$-th coordinate slice
there is no guarantee that the constituent rectangles will interact
productively with their reaches under the application of
$\Lambda_{\widetilde D}$.

No nontrivial result is currently known regarding the high-dimension
($d\geq 2$) bilinear multiplier problem for domains in the local
$L^2$ case, and once again it seems that significantly different
techniques should be used to treat this range of $L^p$ spaces.
\subsection{Remarks}
\begin{enumerate}
\item Of course, as in Theorem \ref{grafakostheorem}, the local
smoothness and curvature assumptions of the Main Theorem are not
necessary \textit{per se}.  One need only guarantee that a
collection of normal vectors to $\partial \widetilde D$
occurring in a coordinate slice yields Besicovitch sets as in
Property \ref{besicovitchproperty}; for a characterization of
such collections, see the paper \cite{bateman} of Bateman.
\item Using (the proof of) the multilinear version of de Leeuw's Theorem proved in
\cite{diestelgrafakos}, one can readily derive analogues of the
Main Theorem for multipliers given by domains $\widetilde D$ in
higher-dimensional spaces, with $\Gamma = \Gamma_2$ replaced by
\[\Gamma_d := \{(\xi_1,\xi_2,\xi_3) \in \BBR^d \times \BBR^d
\times \BBR^d \; | \; \xi_1 + \xi_2 + \xi_3 = 0\}.\]  However,
the curvature conditions arising in this setting are markedly
clumsier.  As in the Main Theorem, the intersection of some
neighborhood in $\partial \widetilde D$ with some $2$-plane of a
prescribed form must be a plane curve of nonzero curvature; the
permissible such $2$-planes are dictated by our conditions and
de Leeuw's Theorem.  Of course, due to the abundance of
nontrivial normal curvature guaranteed by strict convexity,
Corollary \ref{convexcor} holds as stated with $\BBR^2$ replaced
by $\BBR^d$ and $\BBR^4$ replaced by $\BBR^{2d}$.
\item In the same vein, one can generalize the
Main Theorem to a statement about $k$-linear operators (or
$(k+1)$-linear forms); again, the arising curvature conditions
are obtained by slicing a domain in $\BBR^{kd}$ (or its
appropriate embedding into $\BBR^{(k+1)d}$) by prescribed
$2$-planes.  As in the previous remark, Corollary
\ref{convexcor} continues to hold in the general $k$-linear
setting.
\end{enumerate}

\begin{comment}

\end{comment}
\bibliographystyle{alpha}
\bibliography{ballbib}
\end{document}